\documentclass[11pt]{amsart}

\input xy
\xyoption{all}
\usepackage{amsmath,hyperref,amssymb,amsthm,mathrsfs,lscape,xcolor}

\newtheorem{thm}{Theorem}[section]
\newtheorem{cor}[thm]{Corollary}
\newtheorem{lem}[thm]{Lemma}
\newtheorem{prop}[thm]{Proposition}

\theoremstyle{remark}

\newtheorem{rmk}[thm]{Remark}
\newtheorem{definition}[thm]{Definition}

\def\bg{\mathbf{B}}

\def\bP{\overline{\mathcal{P}}}

\def\Det{\mathrm{det}}
\def\Ph{\widehat{\mathscr{P}}}
\def\Sym{\mathrm{S}}
\def\Hpl{\mathcal{H}}
\def\Line{\mathscr{L}}

\def\Aff{\mathrm{Aff}}

\def\P{\mathscr{P}}

\def\lk{\mathrm{lk}}

\def\sgn{\mathrm{sgn}}

\def\V{\mathbf{V}}
\def\RR{\mathrm{R}}

\def\BG{\mathrm B}

\def\Quot{\mathrm{Q}}

\def\Tits{\mathcal{T}}

\def\Dual{\mathbb{D}}

\def\Hom{\mathrm{Hom}}

\def\F{\mathbb F}

\def\st{\mathrm{e}}

\def\reg{\mathrm{reg}}

\def\Q{\mathbf{Q}}

\def\Z{\mathbb Z}
\def\N{\mathbb N}
\def\A{\mathrm A}

\def\Pcal{\mathscr P}
\def\Sq{\mathrm{Sq}}
\def\Ht{\widetilde \H}

\def\GL{\mathrm{GL}}

\def\H{\mathrm{H}}

\def\J{\mathrm{J}}

\def\St{\mathrm{St}}

\def\k{\mathbf{k}}

\renewcommand{\L}{\mathbf{L}}

\begin{document}

\title[]{Stanley-Reisner's ring and the occurrence of the Steinberg representation in the hit problem}
\author[]{Nguyen Dang Ho Hai}

\address{Department of Mathematics, College of Sciences, University of Hue, Vietnam}
%
%

\begin{abstract}
	G. Walker and R. Wood proved that in degree $2^n-1-n$, the space of indecomposable elements of $\F_2[x_1,\ldots,x_n]$, considered as a module over the mod 2 Steenrod algebra, is isomorphic to the Steinberg representation of $\GL_n(\F_2)$. We generalize this result to all finite fields by analyzing certain finite quotients of $\F_q[x_1,\ldots,x_n]$ which come from the Stanley-Reisner rings of some matroid complexes. Our method also shows that the space of indecomposable elements in degree $q^{n-1}-n$ has the dimension equal to that of a complex cuspidal representation of $\GL_n(\F_q)$. As a by product, over the prime field $\F_2$, we give a decomposition of the Steinberg summand of one of these quotients into a direct sum of suspensions of Brown-Gitler modules. This decomposition suggests the existence of a stable decomposition derived from the Steinberg module of a certain topological space into a wedge of suspensions of Brown-Gitler spectra.
\end{abstract}



\email{ndhhai@husc.edu.vn}

\maketitle


%

\section{Introduction}

Let $p$ be a prime and let $\F_q$ denote a finite field of $q=p^s$ elements. 
Let $V$ be an $n$-dimensional vector space over $\F_q$. The symmetric power algebra $\Sym^*(V^*)=\bigoplus_{k\ge 0}\Sym^k(V^*)$ on the dual $V^*$ of $V$ is identified with the polynomial algebra $\F_q[x_1,\ldots,x_n]$ where $\{x_1,\ldots,x_n\}$ is a basis of $V^*$. We assign to each variable $x_i$ degree one, so $\Sym^k(V^*)$ has a basis consisting all monomials of degree $k$. 

The Steenrod algebra $\P$ generated over $\F_q$ by the reduced Steenrod operations $\P^i$, $i\ge 0$, defined as in \cite{LarrySmith}, acts on $\Sym^*(V^*)$ as follows:
\begin{eqnarray*}
\P^k(x^m)&=&\binom{m}{k} x^{m+(q-1)k}, \quad x\in \Sym^1(V^*),\\
\P^k(fg)&=&\sum_{i+j=k}\P^i(f)\P^j(g), \quad f,g\in \Sym^*(V^*). 
\end{eqnarray*}
This action commutes with the action of $\GL_n:=\GL_n(\F_q)$ by linear substitutions of the variables $x_1,\ldots,x_n$. Given $M$ a $\P$-module, the module of $\P$-indecomposable elements of $M$ is defined by $$\Quot(M):= M/\P^+M,$$
$\P^+$ denoting the augmentation ideal of $\P$. 
The Peterson hit problem in algebraic topology, focused mainly on the finite field $\F_2$, asks for a determination of a minimal generating set of the polynomial algebra $\Sym^*(V^*)$ as a module over the Steenrod algebra. 
This is the same as the problem of determining $\Quot(\Sym^*(V^*))$ as a graded vector space, or more generally as a graded module over the group ring $\F_q[\GL_n]$. 
Though much work has been done for this problem, the general answer seems to be
out of reach with the present techniques. The reader is referred to the two recent books of G. Walker and R. M. W. Wood \cite{WalkerWoodBook1,WalkerWoodBook2} for a thorough exposition of this problem.

The starting point of this work is the following result, which is also due to Walker and Wood \cite{WalkerWood2007}: 

\begin{thm}[Walker-Wood]
	For $q=2$, $\Quot^{2^n-1-n}(\Sym^*(V^*))$ is isomorphic to the Steinberg representation of $\GL_n$.
\end{thm}

Walker and Wood proved this by establishing a link between the hit problem over $\F_2$ and Young tableaux. They proved that in some generic degree $d$, the semistandard tableaux can be used to index a generating set for $\Quot^d(\Sym^*(V^*))$. In degree $2^n-1-n$, the hook formula gives the upper bound $\dim_{\F_2}\Quot^{2^n-1-n}(\Sym^*(V^*))\le 2^{\frac{n(n-1)}{2}}$ and the equality then follows from the first occurrence of the Steinberg module in this degree \cite{MP83,KuhnMitchell86}.

The purpose of this paper is to present a new approach to the above result which is valid for all finite fields. The point we discover here is that there is a finite quotient $\RR(V^*)$ of $\Sym^*(V^*)$ such that its top-degree submodule, $\RR^{(q^n-1)/(q-1)-n}(V^*)$, is $\P$-indecomposable and isomorphic to the Steinberg module, and that the natural projection $\Quot^i(\Sym^*(V^*))\twoheadrightarrow \Quot^i(\RR(V^*))$ is an isomorphism in the range $0\le i\le (q^n-1)/(q-1)-n$. A part from classical tools in studying the hit problem, a novel one which we are going to employ in this work is the Stanley-Reisner ring of a simplicial complex \cite{Stanley-book1996}. We note that our method also leads to a formula for the dimension of $\Quot(\Sym^*(V^*))$ in degree $q^{n-1}-n$ and as a by product, for $q=2$, we obtain a decomposition of the Steinberg summand of a variant of $\RR(V^*)$ into a direct sum of Brown-Gitler modules \cite{BG73}.

We now describe the main results of the paper. For this let us choose a nonzero vector $u_\ell$ on each line $\ell$ in $V^*$. 
Given a subspace $W$ of $V^*$, denote by $\Line_W$ the set of all lines in $W$. 
Put $\L_W=\prod_{\ell\in \Line_W} u_\ell$, the ``product of all lines" in $W$. If $W'$ is a subspace of $W$, put $\V_{W',W} =\prod_{\ell\in \Line_{W} \setminus \Line_{W'}} u_\ell$. 
If $\{x_1,\ldots,x_n\}$ is a basis of $V^*$ then it is well-known that $\L_{V^*}$ is (up to a nonzero scalar) the polynomial $$\L_n=\L(x_1,\ldots,x_n):=\det(x_j^{q^{i-1}})_{1\le i,j\le n}$$ introduced by Dickson \cite{Dic11}. Similarly, if $H$ is the hyperplane generated by $x_1,\ldots,x_{n-1}$ then $\V_{H,V^*}$ is (up to a nonzero scalar) the polynomial
$$\V_n=\V(x_1,\ldots,x_n):=\prod_{\lambda_i\in \F_q}(\lambda_1x_1+\cdots+\lambda_{n-1}x_{n-1}+x_n)$$ introduced by M\`ui \cite{Mui75}.

For $k\ge 1$, let $I(V^*,k)$ (or $I_{n,k}$) denote the ideal of $\Sym^*(V^*)$ 
generated by the polynomials $\V_{H,V^*}^k$, where $H$ runs over the set of all hyperplanes of $V^*$. Since $\GL_n$ acts transitively on the set of hyperplanes of $V^*$, it is equivalent to say that $I(V^*,k)$ is generated by the orbit of $\V_n^k$ under the action of $\GL_n$. 
 Let $\RR({V^*,k})$ (or $\RR_{n,k}$) denote the quotient $\Sym^*(V^*)/I(V^*,k)$. It is clear that the natural projection $\Sym^*(V^*)\twoheadrightarrow\RR({V^*,k})$ is both $\P$-linear and $\GL_n$-linear.


The Steinberg representation $\St_n$ \cite{Ste56} is a projective absolutely irreducible representation of $\GL_n$ of dimension $q^{\frac{n(n-1)}{2}}$. It is isomorphic to the right $\F_q[\GL_n]$-module $\st_n\cdot \F_q[\GL_n]$ where $\st_n$ is the Steinberg idempotent defined by 
$$\st_n=\frac{1}{[\GL_n:U_n]}\sum_{b\in B_n,\ \sigma\in \Sigma_n } \sgn(\sigma)b\sigma,$$
$B_n, U_n, \Sigma_n$ denoting respectively the subgroup of upper triangular matrices, the subgroup of upper triangular matrices with 1's on the diagonal, and the symmetric group of permutation matrices. Let $\Det^{i}$, $i\in \Z$, denote the $i$-th power of the determinant representation of $\GL_n$. Then we have $\Det^{i}\cong\Det^j$ if $i\equiv j\mod (q-1)$ and $\Det^{0},\ldots,\Det^{q-2}$ form a complete set of distinct one-dimensional representations of $\GL_n$ over $\F_q$. A twisted Steinberg module is defined to be a tensor product $\St_n\otimes \Det^i$ and this is again a projective absolutely irreducible representation of $\GL_n$ of dimension $q^{\frac{n(n-1)}{2}}$. The idempotent corresponding to $\St_n\otimes \Det^i$ is denoted by $\st_n^{(i)}$ and is given by $\st_n^{(i)}=\phi_i(\st_n)$ where $\phi_i$ is the automorphism of $\F_q[\GL_n]$ given by $\phi_i(g)=\Det^i(g^{-1})g$ (see \cite[\S 2]{Mitchell85}).

\begin{thm} \label{main1}Put $d:=\frac{k(q^n-1)}{q-1}-n$.
	\begin{enumerate}
		\item $\RR(V^*,k)$ is trivial in degree greater than $d$, and in this degree, it is isomorphic to the twisted Steinberg module $\St_n\otimes \Det^{k-1}$.
		\item 	For $k=q^sr$ with $1\le r\le q-1$, the following hold: 
		\begin{enumerate}
			\item For each $0\le i\le d$, the natural projection $\Sym^*(V^*)\twoheadrightarrow \RR(V^*,k)$ induces an isomorphism of $\GL_n$-modules $\Quot^i(\Sym^*(V^*))\cong \Quot^i(\RR(V^*,k))$. 
			\item The top-degree module $\RR^{d}(V^*,k)$ is $\P$-indecomposable.
		\end{enumerate}
	\end{enumerate}
\end{thm}

As a consequence, we obtain the following generalization of the above result of Walker and Wood.
\begin{cor}
	For $k=q^sr$ with $1\le r\le q-1$, there is an isomorphism of $\GL_n$-modules: 
	$$\Quot^{d}(\Sym^*(V^*))\cong \St_n\otimes \Det^{k-1}.$$
\end{cor}
This is because we have a sequence of isomorphisms $\GL_n$-modules: 
$$\St_n\otimes \Det^{k-1} \cong \RR^d(V^*,k) \xrightarrow{\cong} \Quot^{d}(\RR(V^*,k)) \xleftarrow{\cong} \Quot^{d}(\Sym^*(V^*)),$$
where the first isomorphism follows from Part (1) of Theorem \ref{main1} (1), the second from (2.b) and the third from (2.a).

The strategy for the proof of Theorem \ref{main1}(1) is as follows. We consider a simplicial complex $\Delta(V^*,k)$ whose Alexander dual is homotopy equivalent to the Tits building of $\GL_n$. The complex $\Delta(V^*,k)$ is a matroid complex and so the Stanley-Reisner ring $\F_q[\Delta(V^*,k)]$ is Cohen-Macaulay. The quotient module $\RR(V^*,k)$ defined above is then showed to be isomorphic to the quotient of $\F_q[\Delta(V^*,k)]$ by the ideal generated by a well-chosen regular sequence of $\F_q[\Delta(V^*,k)]$. The linear structure of $\RR(V^*,k)$, in particular of its top-degree submodule, will follow from some classical results in the theory of Stanley-Reisner rings \cite{Stanley-book1996}.  
In order to prove Theorem \ref{main1} (2.a), we need to show that all elements of degree $\le d$ in the ideal $I(V^*,k)$ are \textit{hit} (i.e. $\P$-decomposable) in $\Sym^*(V^*)$. For example, when $k=1$, this is proved by using the $\chi$-trick \cite{Wood89} and the following identity (Lemma \ref{Vn}):
$$\V_n=(-1)^{n-1}\sum_{i=0}^{n-1}\chi(\P^{\frac{q^{n-1}-1}{q-1}-i})\big( e_ix_n \big),$$ where  $\chi:\P\to\P$ is the canonical anti-automorphism, and $e_i$ is the $i$-th elementary symmetric function of $x_1^{q-1},\ldots,x_{n-1}^{q-1}$. Theorem \ref{main1} (2.b) will be proved by showing that $\RR(V^*,k)$ can be 
embedded in a direct product of copies of the quotient $\F_q[x_1,\ldots,x_n]/(x_1^{q^{n-1}k},\ldots,x_{n-1}^{qk},x_n^k)$ (which is an example of Poincar\'e duality algebras with trivial Wu classes \cite{MeyerSmith}). 


Our method can also be used to obtain the following: 

\begin{thm}\label{cuspidal}
	$\dim_{\F_q}\Quot^{q^{n-1}-n}(\Sym^*(V^*))=(q-1)(q^2-1)\cdots (q^{n-1}-1),$ $n\ge 2$. 
\end{thm}

This will be proved by considering a simplicial complex whose Alexander dual is homotopy equivalent to a complex employed by G. Lusztig \cite{Lusztig74} in his work on constructing discrete series (a.k.a. cuspidal representations) of $\GL_n$.





To state the final result of the paper, suppose $q=k=2$ and denote by $\A$ the mod $2$ Steenrod algebra. Using the work of Inoue \cite{Ino02} on the hit problem of the Steinberg summand, we will show that the Steinberg summand of $\RR(V^*,2)$ is decomposed into a direct sum of suspensions of Brown-Gitler modules \cite{BG73}. 
\begin{thm}\label{decomposition} Put $d=2(2^n-1)-n=\sum_{j=1}^n(2^j-1)$.
	There is an isomorphism of $\A$-modules
	$$\bigoplus_{j=0}^{n} \Sigma^{d-(2^j-1)}\BG(2^j-1) \xrightarrow{\cong} \RR_{n,2}\cdot \st_n.$$
	Here $\BG(k)$ denotes the Brown-Gitler module 
	$\BG(k):=\A/\A\langle \chi(\Sq^i) \mid 2i>k\rangle$ and $\RR_{n,2}\cdot \st_n$ denotes the direct summand of $\RR_{n,2}$ associated to the Steinberg idempotent $\st_n$.
\end{thm}

In an ongoing project, by using the work \cite{DavidJanuskewiz91} of M. W. Davis and T. Januszkiewicz, we interpret $\RR(V^*,k)$ as the mod $2$ cohomology of the orbit space $\rho(V^*,k)_\reg/V$, where $\rho(V^*,k)$ denotes the direct sum of $k$ copies of the reduced regular representation of $V$ and $\rho(V^*,k)_\reg$ denotes the regular part of the action of $V$ on $\rho(V^*,k)$. The decomposition in Theorem \ref{decomposition} suggests that there would exist a homotopy equivalence (of $2$-completed spectra)
$$\st_n\cdot \Sigma^\infty(\rho(V^*,2)_\reg/V)_+\simeq \bigvee_{j=0}^n \Sigma^{d-(2^j-1)} \bg(2^j-1),$$
$\bg(k)$ denoting the Brown-Gitler spectrum whose cohomology is $\BG(k)$. We intend to go back to this problem in a future work. 



\tableofcontents

\section{The Stanley-Reisner ring of a simplicial complex} \label{SR-ring}

In this section, we review some facts about the Stanley-Reisner ring of a simplicial complex which will be used in proving Theorem \ref{main1} (1). Our main reference is \cite{Stanley-book1996}. 

Let $S$ be a non-empty finite set. A \textit{simplicial complex} $\Delta$ on the vertex set $S$ is a collection of subsets of $S$ that is closed under inclusion. A subset $F$ of $S$ is called a \textit{face} of $\Delta$ if $F\in \Delta$ and is called a \textit{non-face} of $\Delta$ if $F\not\in \Delta$.
The \textit{dimension} of a face $F$ is $\dim F:=|F|-1$ and the \textit{dimension} of $\Delta$ is $\dim\Delta:=\max_{F\in \Delta} \dim F$.  
A maximal face under inclusion is called a \textit{facet} and $\Delta$ is said to be \textit{pure} if all facets have the same dimension.

Let $\k$ be a field and let $\k[X_s\mid s\in S]$ denote the polynomial algebra over $\k$ in the variables $X_s$ indexed by the vertex set $S$. 
Given a subset $F$ of $S$, let $X_F$ denote the monomial $\prod_{s\in F}X_s$.  
The \textit{face ring} (or the \textit{Stanley-Reisner ring}) $\k[\Delta]$ of $\Delta$ is then defined to be the quotient ring 
$$\k[\Delta]:=\k[X_s| s\in S]/I_\Delta,$$
where $I_\Delta$ is the \textit{Stanley-Reisner ideal} defined by
$I_\Delta:=(X_F| F\not \in \Delta).$ It is clear that $I_\Delta$ can be defined by using only monomials corresponding to minimal non-faces of $\Delta$.

Next recall that the Hilbert-Poincar\'e series of an $\N$-graded $\k$-vector space $M$ (with $\dim_\k M^i$ finite for all $i$) is defined by 
$P(M,t):=\sum_{i\ge 0} (\dim_\k M^i) t^i.$
Suppose that $\dim \Delta =d-1$ and regard the variables $X_s$ as being of degree one. The face ring $\k[\Delta]$ is then a graded ring and its Hilbert-Poincar\'e series is given by: 
\begin{equation}
P(\k[\Delta],t)=\sum_{i=-1}^{d-1}\frac{f_it^{i+1}}{(1-t)^{i+1}},
\end{equation}
where $f_{-1}=1$ and $f_i:=f_i(\Delta)$, $i\ge 0$, is the number of $i$-dimensional faces of $\Delta$  (\cite[Theorem 1.4]{Stanley-book1996}). The sequence $(f_0,\ldots,f_{d-1})$ is called the \textit{$f$-vector} of $\Delta$. Writing $P(\k[\Delta],t)$ in the form 
\begin{equation}
P(\k[\Delta],t)=\frac{h_0+h_1t+\cdots+h_dt^d}{(1-t)^d},
\end{equation}
then the sequence $(h_0,\ldots,h_d)$ is called the \textit{$h$-vector} of $\Delta$. In terms of the $f$-vector, each $h_k:=h_k(\Delta)$ is given by 
$h_k=\sum_{i=0}^{k}(-1)^{k-i} \binom{d-i}{k-i}f_{i-1}.$
In particular, 
$h_d=\sum_{i=0}^d(-1)^{d-i}f_{i-1}$ and so 
\begin{equation}
h_d(\Delta)=(-1)^{d-1}\tilde \chi (\Delta),
\end{equation}
where $\tilde \chi (\Delta)=\sum_{i\ge -1}(-1)^if_i=\sum_{i\ge -1}(-1)^i\dim_\k\Ht_i(\Delta;\k)$ is the reduced Euler characteristic of $\Delta$. Here $\Ht_i(\Delta;\k)$ denotes the reduced homology of $\Delta$ with coefficients in $\k$ (\cite[Definition 3.1]{Stanley-book1996}). 
We will also need to express $h_d$ in terms of the reduced characteristic of the Alexander dual of $\Delta$. Recall that the \textit{Alexander dual} $\Delta^*$ of $\Delta$ is the simplicial complex with the same vertex set $S$ such that a subset $F$ of $S$ is a face of $\Delta^*$ if the complement $S\setminus F$ is not a face of $\Delta$. Since for each $0\le i\le |S|$, the number of subsets of cardinality $i$ of $S$ is equal to $f_{i-1}(\Delta)+f_{|S|-i-1}(\Delta^*)$ and the alternating sum of these numbers is zero, it follows that $\tilde \chi (\Delta^*) +(-1)^{|S|} \tilde \chi (\Delta^*)=0$, and so  
\begin{equation}
h_d(\Delta)=(-1)^{|S|-d}\tilde \chi(\Delta^*).
\end{equation}

We review now some commutative algebra. Let $R$ be an $\N$-graded algebra over $\k$. The \textit{Krull dimension} of $R$, denoted $\dim R$, is the maximal number of algebraically independent homogeneous elements of $R$. Suppose $\dim R=d$. A sequence $(\theta_1,\ldots,\theta_d)$ of homogeneous elements of $R$ is called a \textit{homogeneous system of parameters} (\textit{h.s.o.p.}) if $R/(\theta_1,\ldots,\theta_d)$ is a finite-dimensional $\k$-vector space. Equivalently, $(\theta_1,\ldots,\theta_d)$ is a h.s.o.p if $\theta_1,\ldots,\theta_d$ are algebraically independent and $R$ is a finitely-generated $\k[\theta_1,\ldots,\theta_d]$-module. A h.s.o.p. $(\theta_1,\ldots,\theta_d)$ is called a \textit{linear system of parameters} (l.s.o.p.) if each $\theta_i$ is of degree one.

The Krull dimension of the face ring $\k[\Delta]$ is $1+\dim \Delta$ (\cite[Theorem 1.3]{Stanley-book1996}). Note that the existence of a l.s.o.p for $\k[\Delta]$ is assured by Noether's Normalization Lemma when $\k$ is infinite. When $\k$ is finite, one may need to pass to an infinite extension field which does not affect the Cohen-Macaulay property. 
We need the following result to recognize a l.s.o.p. in $\k[\Delta]$. 

\begin{lem}[{\cite[Lemma 2.4]{Stanley-book1996}}] \label{lsop} 
	\begin{enumerate}
		\item Let $\k[\Delta]$ be a face ring of Krull dimension $d$ and let $\theta_1,\ldots,\theta_d\in \k[\Delta]^1$. Then $(\theta_1,\ldots,\theta_d)$ is a l.s.o.p. for $\k[\Delta]$ if and only if for every face $F$ of $\Delta$, the restrictions ${\theta_1}_{|_F},\ldots,{\theta_d}_{|_F}$ span a vector space of dimension equal to $|F|$. Here the restriction $\theta_{|_F}$ of an element $\theta=\sum_{s\in S}\alpha_s X_s$ of $\k[\Delta]^1$ to a face $F$ is defined by $\theta_{|_F}=\sum_{s\in F}\alpha_s X_s$.
		\item If $(\theta_1,\ldots,\theta_d)$ is a l.s.o.p. for $\k[\Delta]$, then the quotient ring $\k[\Delta]/(\theta_1,\ldots,\theta_d)$ is spanned as a $\k$-vector space by the monomials $X^F$, $F\in \Delta$.
	\end{enumerate}
\end{lem}

Next recall that a sequence $(\theta_1,\ldots,\theta_d)$ of homogeneous elements of an $\N$-graded $\k$-algebra $R$ is called a \textit{regular sequence} if $\theta_{i+1}$ is not a zero-divisor in $R/(\theta_1,\ldots,\theta_i)$ for $0\le i<d$. Equivalently, $(\theta_1,\ldots,\theta_d)$ is a regular sequence if $\theta_1,\ldots,\theta_d$ are algebraically independent and $R$ is a finite-dimensional free $\k[\theta_1,\ldots,\theta_d]$-module. The $\k$-algebra $R$ is \textit{Cohen-Macaulay} if it admits a regular sequence. It is known that if $R$ is Cohen-Macaulay then any h.s.o.p. is a regular sequence. 
If the face ring $\k[\Delta]$ is Cohen-Macaulay, then we say that $\Delta$ is a \textit{Cohen-Macaulay complex} (over $\k$). A fundamental result of G. A. Reisner \cite{Reisner} states that $\Delta$ is Cohen-Macaulay over $\k$ if and only if for all $F\in \Delta$ and all $i<\dim (\lk F)$, we have $\Ht^i(\lk F; \k)=0$. Here $\lk F$ is the \textit{link} of $F$ defined by $\lk F=\{G\in \Delta| G\cup F\in \Delta, \ G\cap F=\emptyset \}$. 

Matroid complexes provide a rich source of Cohen-Macaulay complexes. 
A simplicial complex $\Delta$ on the vertex set $S$ is a \textit{matroid complex} if it satisfies the \textit{exchange property}: given $F, G\in\Delta$ with $|F|<|G|$, there exists $v\in  G\setminus F$ such that $F\cup \{v\}\in \Delta$. This is the property satisfied by the independent sets of a matroid \cite{Oxley}. It is known that a matroid complex is pure and for every linear order of $S$, the induced lexicographical order on the facets of the matroid complex $\Delta$ make it into a shellable complex \cite[Theorem 7.3.3]{Bjorner92}. 
Recall that a \textit{shellable complex} is a pure simplicial complex $\Delta$ together with an order of its facets $F_1,\ldots, F_s$ such that for each $1\le i\le s$, the faces of $\Delta_i$ which do not belong to $\Delta_{i-1}$ has a unique element (denoted by $r(F_i)$) with respect to inclusion (\cite[Definition 2.1]{Stanley-book1996}). Here $\Delta_0=\emptyset$ and $\Delta_i$ is the subcomplex generated by $F_1,\ldots,F_i$ defined by $\Delta_i:=2^{F_1}\cup \cdots \cup 2^{F_i}$, where $2^F:=\{G: G\subset F\}$. 
A shellable complex is Cohen-Macaulay over any field; furthermore, if $(\theta_1,\ldots,\theta_d)$ is a l.s.o.p. for $\k[\Delta]$, then $\k[\Delta]$ is a free $\k[\theta_1,\ldots,\theta_d]$-module with basis $\{X_{r(F_i)}: 1\le i\le s\}$ (\cite[Theorem 1.7, Corollary 1.8]{Bjorner84} or \cite[Theorem 2.5]{Stanley-book1996}).

%

%


We summarize the above discussions in the following proposition which will be used in the next sections. For this we identify the polynomial algebra $\k[X_s| s\in S]$ with the symmetric power algebra $\Sym^*(\k^S)$ where $\k^S$ denotes the $\k$-vector space of functions from $S$ to $\k$. The dual of $\k^S$ is the space $\k\langle S\rangle$ of formal sums $\sum_{s\in S}\alpha_s Y_s$ with $\alpha_s\in \k$ where $Y_s$ is dual to $X_s$. For each monomorphism of vector spaces $\iota: V\hookrightarrow \k\langle S\rangle$, let $\iota^*(I_\Delta)$ denote the image of the Stanley-Reisner ideal $I_\Delta$ under the induced epimorphism $\iota^*: \Sym^*(\k^S)\twoheadrightarrow \Sym^*(V^*)$. So $\iota^*(I_\Delta)$ is the ideal of $\Sym^*(V^*)$ generated by the polynomials $\iota^*(X^F)=\prod_{s\in F}\iota^*(X_s)$, $F\in \Delta$.

\begin{prop}\label{lsop-dual}
	Suppose $\dim_\k V=|S|-d$ and $\iota: V\hookrightarrow\k\langle S\rangle$ is a monomorphism which satisfies the property:
	
	{\rm (*)} For each face $F$ of $\Delta$, the intersection of $\iota(V)$ with the subspace $\k\langle F  \rangle$ of $\k\langle S\rangle$ is trivial. 
	
	\noindent Then $\Sym^*(V^*)/\iota^*(I_\Delta)$ is spanned as a $\k$-vector space by the $\iota^*(X^F)$, $F\in \Delta$. In particular, the top-degree subspace $(\Sym^*(V^*)/\iota^*(I_\Delta))^d$ is spanned as a $\k$-vector space by the set $\{\iota^*(X^F)|\text{$F$ is facet of $\Delta$}\}$. Furthermore, if $\Delta$ is Cohen-Macaulay over $\k$ (in particular if $\Delta$ is a matroid complex), then the Hilbert-Poincar\'e series of $\Sym^*(V^*)/\iota^*(I_\Delta)$ is $h_0+h_1t+\cdots+h_dt^d$, where $h_d$ can be computed by $h_d=(-1)^{d-1}\tilde{\chi}(\Delta)=(-1)^{|S|-d}\tilde \chi(\Delta^*).$ 
\end{prop}

\begin{proof}
	Let $\lambda:\k\langle S\rangle \twoheadrightarrow W$ be the cokernel of $\iota$. So $\dim_\k W=d$ and there is a short exact sequence of vector spaces $0\to V\xrightarrow{\iota} \k\langle S\rangle \xrightarrow{\lambda} W\to 0$. For each face $F$ of $\Delta$, the triviality of $\iota(V)\cap \k\langle F\rangle$ implies the injectivity of the composition $\k\langle F\rangle \hookrightarrow \k \langle S\rangle \xrightarrow{\lambda} W$.  Taking the dual of this we see that the composition $W^* \xrightarrow{\lambda^*} \k^S \twoheadrightarrow \k^F$ is surjective. Note that $\k^S \twoheadrightarrow \k^F$ is the restriction map $\theta\mapsto \theta|_F$ defined in Lemma \ref{lsop} above. It follows that if we choose a basis ${w_1,\ldots,w_d}$ of $W^*$ and put $\theta_i=\lambda^*(w_i)$, then ${\theta_1}_{|_F},\ldots,{\theta_d}_{|_F}$ span $\k^F$, and so by Proposition \ref{lsop}(1), $(\theta_1,\ldots,\theta_d)$ is a l.s.o.p. for $\k[\Delta]$.  The map $\iota$ now induces an isomorphism of $\k$-algebras:
	\begin{equation}
	\k[\Delta]/(\theta_1,\ldots,\theta_d)\cong \k\otimes_{\Sym^*(W^*)} (\Sym^*(\k^S)/I_\Delta)\cong \Sym^*(V^*)/\iota^*(I_\Delta).
	\end{equation} 
	By Proposition \ref{lsop}(2), this isomorphism implies that the $\k$-vector space $\Sym^*(V^*)/\iota^*(I_\Delta)$ is spanned by the $\iota^*(X^F)$, $F\in \Delta$. If $\k[\Delta]$ is Cohen-Macaulay, then $\k[\Delta]$ is free over $\k[\theta_1,\ldots,\theta_d]$ and so the Hilbert-Poincar\'e series of $\k[\Delta]/(\theta_1,\ldots,\theta_d)$ is equal to $(1-t)^dP(\k[\Delta],t)=h_0+h_1t+\cdots+h_dt^d$. 
\end{proof}



\section{$\RR(V^*,k)$ and the Steinberg module} \label{R-Steinberg}

Let $p$ be a prime and let $\F_q$ denote a finite field of $q=p^s$ elements. 
Let $V$ be an $n$-dimensional vector space over $\F_q$. Given a linear vector space $W$, denote by $\Line_W$ the set of lines (i.e. $1$-dimensional linear subspaces) in $W$ and by $\Hpl_W$ the set of hyperplanes (i.e. $1$-codimensional linear subspaces) in $W$.
A set of lines $\{\ell_1,\ldots,\ell_m\}$ in $W$ is called an \textit{$m$-frame} of $W$ if $\{\ell_1,\ldots,\ell_m\}$ linearly spans an $m$-dimensional subspace of $W$.

\begin{definition} Let $\Delta(V^*)$ be the simplicial complex on the vertex set $\Line:=\Line_{V^*}$ in which a face is of the form $\Line\setminus F$ where $F$ is a set of lines which linearly spans $V^*$. More generally, for each $k\ge 1$, let $\Delta(V^*,k)$ be the simplicial complex on  $\Line\sqcup \cdots \sqcup \Line$ (the disjoint union of $k$ copies of $\Line$) in which a face is of the form $(\Line\setminus F_1)\sqcup\cdots \sqcup (\Line\setminus F_k)$ where the union $F_1\cup \cdots \cup F_k$ linearly spans $V^*$. 
\end{definition}

It is straightforward to check that a facet of $\Delta(V^*)$ is of the form $\Line\setminus F$ where $F$ is an $n$-frame of $V^*$ and a minimal non-face of $\Delta(V^*)$ is of the form $\Line\setminus \Line_H$ where $H$ is a hyperplane of $V^*$.
Similarly a facet of $\Delta(V^*,k)$ is of the form $(\Line\setminus F_1)\sqcup\cdots \sqcup (\Line\setminus F_k)$ where $F_1,\ldots,F_k$ form a partition of a frame of $V^*$ and a minimal non-face of $\Delta(V^*,k)$ is of the form $(\Line\setminus \Line_H) \sqcup\cdots \sqcup (\Line\setminus \Line_H)$ where $H$ is a hyperplane of $V^*$. The cardinality of a facet of $\Delta(V^*,k)$ is $$d:=k(q^n-1)/(q-1)-n,$$ and so $\Delta(V^*,k)$ is a $(d-1)$-dimensional simplicial complex. 



\begin{lem}\label{Deltamatroid}
	$\Delta(V^*,k)$ is a matroid complex.
\end{lem}

\begin{proof} Recall that the \textit{dual} of a matroid complex $\Delta$ on $S$ is defined to be the simplicial complex $\Delta^{\#}$ on $S$ such that $F$ is a facet of $\Delta^\#$ if $S\setminus F$ is a facet of $\Delta$. It is known that the dual of a matroid complex is also a matroid complex \cite[Theorem 2.1.1]{Oxley}. 
	
Let $\nabla(V^*,k)$ denote the simplicial complex on $\Line\sqcup \cdots \sqcup \Line$ in which $F_1\sqcup \cdots \sqcup F_k$ is a face if $F_1,\ldots,F_k$ form a partition of a frame of $V^*$. The exchange property is easily checked for $\nabla(V^*,k)$ and so it is a matroid complex. Since a facet of $\nabla(V^*,k)$ is of the form $F_1\sqcup \cdots \sqcup F_k$ where $F_1,\ldots,F_k$ form a partition of an $n$-frame of $V^*$, it follows that the dual of the matroid complex $\nabla(V^*,k)$ is $\Delta(V^*,k)$, and so $\Delta(V^*,k)$ is also a matroid complex. 
\end{proof}


The space
$\F_q\langle \Line\sqcup \cdots \sqcup \Line\rangle$ is the space of formal sums $\sum_{1\le i\le k, \ \ell\in \Line}\alpha_{\ell,i} Y_{\ell,i}$ with $\alpha_{\ell,i}\in \F_q$, where $Y_{\ell,i}$ corresponds to $\ell\in \Line$ with $\Line$ being at the $i$-th position in the disjoint union $\Line\sqcup \cdots \sqcup \Line$.
For each line $\ell$ in $V^*$, choose a nonzero linear form $u_\ell$ on it. Such a choice gives rise to a homomorphism 
$\iota:V \to\F_q\langle  \Line\sqcup \cdots \sqcup \Line \rangle$ defined by 
$$\iota(v)=\sum_{1\le i\le k,\ \ell\in \Line} u_\ell(v)Y_{\ell,i}.$$

\begin{lem} \label{iota-property}
	The map $\iota$ is a monomorphism satisfying the property (*) and its dual $\iota^*:\F_q^{\Line\sqcup \cdots \sqcup \Line} \to V^*$ sends each basis element $X_{\ell,i}$ to $u_\ell$. 
\end{lem}

\begin{proof}That $\iota$ is a monomorphism is clear. Given $(\Line\setminus F_1)\sqcup\cdots \sqcup (\Line\setminus F_k)$ a face of $\Delta(V^*,k)$, if $\iota(v)=\sum_{1\le i\le k,\ \ell\in \Line} u_\ell(v)Y_{\ell,i}$ belongs to $\F_q\langle (\Line\setminus F_1)\sqcup\cdots \sqcup (\Line\setminus F_k)  \rangle$, then we must have $u_\ell(v)=0$ for all $\ell\in F_1\cup \cdots \cup F_k$, which implies that $v=0$ since $F_1\cup \cdots \cup F_k$ linearly spans $V^*$. For the second assertion, we have   $$\iota^*(X_{\ell,i})(v)=X_{\ell,i}(\iota(v))=X_{\ell,i}\bigg(\sum_{1\le i\le k,\ \ell\in \Line} u_\ell(v)Y_{\ell,i}\bigg)=u_\ell(v)$$
	for all $v\in V$. 
\end{proof}

%
%
%
%
%

We now prove the following which is the first part of Theorem \ref{main1}.
\begin{prop} 
	$\RR(V^*,k)$ is trivial in degree greater than $d$, and in this degree, it is isomorphic to the twisted Steinberg module $\St_n\otimes \Det^{k-1}$.
\end{prop}

\begin{proof} Recall that a minimal non-face of $\Delta(V^*,k)$ is of the form $(\Line\setminus \Line_H) \sqcup\cdots \sqcup (\Line\setminus \Line_H)$ where $H$ is a hyperplane of $V^*$. The generator of the Stanley-Reisner ideal $I_{\Delta(V^*,k)}$ of $\F_q[X_{\ell,i}\mid  \ell\in\Line, \ 1\le i\le k]$ corresponding to such a minimal non-face is $\prod_{1\le i\le k,\ \ell\in \Line \setminus \Line_H} X_{\ell,i}$. By Lemma \ref{iota-property}, this is sent by $\iota^*$ to $\big(\prod_{\ell\in \Line \setminus \Line_H} u_\ell\big)^k$ in $\Sym^*(V^*)$. It follows that the ideal $\iota^*(I_{\Delta(V^*,k)})$ of $\Sym^*(V^*)$ is generated by the polynomials $\big(\prod_{\ell\in \Line \setminus \Line_H} u_\ell\big)^k$ where $H$ runs over the set of hyperplanes of $V^*$. The quotient ring $\Sym^*(V^*)/\iota^*(I_{\Delta(V^*,k)})$ is thus exactly the ring $\RR(V^*,k)$ defined in the introduction. 
	
Since $\Delta(V^*,k)$ is a matroid complex (Lemma \ref{Deltamatroid}) and $\iota$ satisfies the property (*) (Lemma \ref{iota-property}), we can apply Proposition \ref{lsop-dual} to analyze the top-degree module $\RR(V^*,k)^d$. Given a facet of $\Delta(V^*,k)$ 
of the form $F:=(\Line\setminus F_1)\sqcup\cdots \sqcup (\Line\setminus F_k)$ where $F_1,\ldots,F_k$ form a partition of an $n$-frame $\{\ell_1,\ldots,\ell_n\}$ of $V^*$, the corresponding generator $\iota^*(X_F)$ of $\RR(V^*,k)^d$ is given by $$\iota^*(X_F)=\frac{\big(\prod_{\ell\in \Line} u_\ell\big)^k}{u_{\ell_1}\cdots u_{\ell_n} }.$$ It is well-known that the ``product of lines" $\prod_{\ell\in  \Line} u_\ell$ is up to a nonzero scalar equal to $\L_n$ where $$\L_n=\L(x_1,\ldots,x_n):=\det(x_j^{q^{i-1}})_{1\le i,j\le n}.$$ Since $\GL_n$ acts transitively on the set of $n$-frames of $V^*$, it follows that if we put 
$\gamma:=\frac{\L_n^k}{x_1\cdots x_n}$ then by Proposition \ref{lsop-dual}, $\{\gamma\cdot g\mid g\in \GL_n \}$ is a spanning set of $\RR(V^*,k)^d$. But it is known that (\cite[Corollary A.7]{Mitchell85}) $\gamma$ is fixed by $\st_n^{(k-1)}$, where $\st_n^{(k-1)}$ is the idempotent corresponding to the twisted Steinberg representation $\St_n\otimes \det^{k-1}$.  
It follows that the $\GL_n$-linear map $$\St_n\otimes {\Det}^{k-1}\cong \st_n^{(k-1)}\cdot \F_q[\GL_n]\to \RR(V^*,k)^d,\quad \st_n^{(k-1)}\cdot g\mapsto \gamma\cdot g,$$ is an epimorphism. 
To conclude that this is an isomorphism, we need to show that the dimension of $\RR(V^*,k)^d$ is $q^{\frac{n(n-1)}{2}}$. For this we use the formula $h_d=(-1)^n\tilde{\chi}(\Delta(V^*,k)^*)$ and proceed as follows. 
	Recall \cite{Wachs} first that to every poset (partially ordered sets) $P$, one can associate a simplicial complex $c(P)$, called the \textit{order complex} of $P$, whose faces are the chains  (i.e. totally ordered subsets) of $P$. The (reduced) homology of $P$ is defined to be the (reduced) homology of $c(P)$. Inversely, to every simplicial complex $\Delta$, one can associate a poset $p(\Delta)$, called the \textit{face poset} of $\Delta$, which is defined to be the poset of nonempty faces ordered by inclusion. It is known that $c(p(\Delta))$, called the \textit{barycentric
		subdivision} of $\Delta$, and $\Delta$ have the same geometric realizations. 
	Now let $\Tits$ denote the Tits building which is defined by   
	$$\Tits:=\text{partially ordered sets (poset) of proper subspaces of $V^*$}.$$
	It is well-known that $\Ht_*(\Tits;\F_q)$ is only nontrivial in degree $n-2$, and in this degree, $\Ht_{n-2}(\Tits;\F_q)\cong \St_n$ (Solomon and Tits). Note that the dimension of $\Ht_{n-2}(\Tits;\F_q)$, which is $q^{\frac{n(n-1)}{2}}$, can be computed by induction using an exact sequence as in \cite[Theorem 1.14]{Lusztig74}.  	
	The Alexander dual $\Delta^*:=\Delta(V^*,k)^*$ is related to the Tits building as follows. 
	By definition, $\Delta^*$ is the simplicial complex on $\Line\sqcup \cdots \sqcup \Line$ in which a face is of the form $F_1\sqcup\cdots \sqcup F_k$ where the union $F_1\cup \cdots \cup F_k$ does not linearly spans $V^*$. Consider the map of posets $f:p(\Delta^*)\to \Tits$ which sends a face $F_1\sqcup \cdots \sqcup F_k$ to the subspace spanned by $F_1\cup \cdots \cup F_k$. For each proper subspace $W$ of $V^*$, the poset $f^{-1}(\Tits_{\le W})$ is contractible because it has $\Line_W \sqcup \cdots \sqcup \Line_W$ as its unique maximal element. The Quillen Fiber Lemma \cite{Quillen-FiberLemma} then implies that $f$ is a homotopy equivalence. The identity $h_d=q^{\frac{q(q-1)}{2}}$ follows. 	
\end{proof}

\section{$\RR(V^*,k)$ as a $\P$-module} \label{R-Steenrod}

In this section we prove the second part of Theorem \ref{main1}. 
For this we first review some facts about the action of the Steenrod algebra $\P$ on $\Sym^*(V^*)\cong \F_q[x_1,\ldots,x_n]$. Larry Smith \cite{LarrySmith} defined the Steenrod algebra $\P:=\P(\F_q)$ as the $\F_q$-subalgebra of the endomorphism algebra of the functor $V\mapsto \Sym^*(V^*)$, generated by certains natural transformations $1=\P^0,\P^1,\P^2,\ldots$. It is sufficient for us to know that the action of the operations $\P^i$ on $\Sym^*(V^*)$ satisfies: 
\begin{itemize}
	\item \textit{the unstable condition}: $\P^i(f)=f^q$ if $i=\deg f$ and $\P^i(f)=0$ if $i>\deg f$, $\forall f\in \Sym^*(V^*)$, and
	\item \textit{the Cartan formula}: $\P^k(fg)=\sum_{i+j=k} \P^i(f)\P^j(g), \forall f,g\in \Sym^*(V^*).$ 
\end{itemize}
Denote by $\bP$ the total Steenrod operation $1+\P^1+\P^2+\cdots$. These properties then implies that the map $$\bP: \F_q[x_1,\ldots,x_n]\to \F_q[x_1,\ldots,x_n], \quad f\mapsto \sum_{i\ge 0}\P^i(f),$$ is a homomorphism of algebras which satisfies $\bP(x)=x+x^q$ for all linear form $x$. 

The algebra $\P$ is also a Hopf algebra with the coproduct $\Delta:\P\to \P\otimes\P$ defined by $$\Delta(\P^m)=\sum_{i+j=m} \P^i\otimes \P^j,\quad m\ge 1.$$ The canonical anti-automorphism (the antipode) $\chi:\P\to \P$ then satisfies the relations: 
\begin{equation*}\label{chi}
\sum_{i+j=m} \P^i\chi(\P^j)=\sum_{i+j=m} \chi(\P^i) \P^j=0, \quad m\ge 1.
\end{equation*}
This can be used to show that $$\chi(\P^i)(x)=\begin{cases}
(-1)^i x^{q^r} & \text{if $i=\frac{q^r-1}{q-1}$,}\\ 0 & \text{otherwise.} 
\end{cases}$$ 
Since $\chi$ is a map of coalgebras, the Cartan formula also holds for $\chi(\P^k)$:
$$\chi(\P^k)(f\cdot g)=\sum_{i+j=k} \chi(\P^i)(f)\cdot \chi(\P^j)(g), \quad \forall f,g\in \Sym^*(V^*).$$
In order to simplify signs, we write $\Ph^i$ for $(-1)^i\chi(\P^i)$ and let $\Ph=1+\Ph^1+\Ph^2+\cdots$ denote the (signed) total conjugate Steenrod operation. The above formulae for $\chi$ then implies that the map $$\Ph:\F_q[x_1,\ldots,x_n]\to \F_q[[x_1,\ldots,x_n]], \quad f\mapsto \sum_{i\ge 0} \Ph^i(f),$$ is a homomorphism of algebras which satisfies $\Ph(x)=x+x^q+x^{q^2}+\cdots$ for all linear form $x$. Here $\F_q[[x_1,\ldots,x_n]]$ denotes the $\F_q$-algebra of power series in $x_1,\ldots,x_n$. 

Finally, for all $k\ge 1$ and $f,g\in \Sym^*(V^*)$, we recall the following congruence
 $$\P^k(f)\cdot g\equiv f\cdot \chi(\P^k)(g)\pmod {\P^+\Sym^*(V^*)},$$
which is known as the the $\chi$-trick \cite{Wood89}. This follows from the identity
\begin{eqnarray*}
\sum_{i=0}^k\P^i\big(f\cdot \chi(\P^{k-i}(g))&=& \sum_{i=0}^k\sum_{j=0}^{i}\P^j(f)\cdot \P^{i-j}(\chi(\P^{k-i}(g))\\
&=&\sum_{j=0}^{k} \bigg(\P^j(f)\cdot \sum_{a+b=k-j} \P^a(\chi(\P^b)(g))\bigg)\\
&=& \P^k(f)\cdot g.
\end{eqnarray*}


We now prove Theorem \ref{main1} (2.a). We suppose $k=q^sr$ with $s\ge 0$, $1\le r\le q-1$ and as in the previous section, we put  $d:=\frac{k(q^n-1)}{q-1}-n$.

\begin{prop}[Theorem \ref{main1} (2.a)] \label{Rmod} For each $0\le i\le d$, the natural projection
	$\Sym^*(V^*)\twoheadrightarrow \RR(V^*,k)$ induces an isomorphism of $\GL_n$-modules $\Quot^i(\Sym^*(V^*))\cong \Quot^i(\RR(V^*,k))$.  
\end{prop}

 We need the following lemma whose proof will be given below.
\begin{lem}\label{Vn} Given $r$ linear forms $y_1,\ldots,y_r$ in $V^*$, the following identity holds  $$\V(x_1,\ldots,x_{n-1},y_1)^{q^s}\cdots \V(x_1,\ldots,x_{n-1},y_r)^{q^s}	=(-1)^{n-1}\sum_{i=0}^{n-1}\chi(\P^{\frac{(q^{n-1}-1)q^sr}{q-1}-i})\big( e_i\cdot   y_1^{q^s}\cdots y_r^{q^s} \big),$$
	where $e_i:=e_i(x_1^{q-1},\ldots,x_{n-1}^{q-1})$ is the $i$-th elementary symmetric function of $x_1^{q-1},\ldots,x_{n-1}^{q-1}$.
\end{lem}

\begin{proof}[Proof of Proposition \ref{Rmod}]
	It is sufficient to prove that $f\cdot\V_n^k$ is $\P$-decomposable in $\Sym^*(V^*)$ if $$\deg(f)\le d-\deg(\V_n^k)=d-kq^{n-1}=\frac{k(q^{n-1}-1)}{q-1}-n.$$ By Lemma \ref{Vn} and $\chi$-trick, we have 
	\begin{eqnarray*}
		f\cdot \V_n^k &=& (-1)^{n-1}\sum_{i=0}^{n-1}f\cdot \chi(\P^{\frac{(q^{n-1}-1)k}{q-1}-i})\big( e_i\cdot  x_n^k \big) \\
		&\equiv& (-1)^{n-1}\sum_{i=0}^{n-1} e_i\cdot  x_n^k \cdot \P^{\frac{(q^{n-1}-1)k}{q-1}-i}(f) \pmod {\Pcal^+\Sym^*(V^*)}.
	\end{eqnarray*}
By instability, $\P^{\frac{(q^{n-1}-1)k}{q-1}-i}(f)=0$ for all $0\le i\le n-1$. 
\end{proof}

We prove Theorem \ref{main1} (2.b) by proving the following:

\begin{prop} \label{emb}
	\begin{enumerate}
		\item $\RR_{n,k}$ can be embedded as a $\P$-submodule into a direct product of $\frac{q^n-1}{q-1}$ copies of $\RR_{n-1,qk}\otimes \RR_{1,k}$.
		\item For all $u\in \RR_{n,k}$, we have $\chi(\P^i)(u)=0$ whenever $|u|+qi>d$. In particular, the top-degree module $\RR_{n,k}^{d}$ is $\P$-indecomposable.
	\end{enumerate}
	\end{prop}

In order to prove this proposition, we need the following lemma whose proof will be again given later.

\begin{lem} \label{ideal-rel}
	For each $0\not= y \in V^*$, 
	$I(V^*,k)+(y^k) = I(W,qk) + (y^k)$ where $W$ is linear complement of $\F_q\langle y\rangle$ in $V^*$. 
\end{lem}

\begin{proof}[Proof of Proposition \ref{emb}] The kernel of the natural map of $\P$-modules $\Sym^*(V^*)\to \prod_{\ell\in \Line_{V^*}} \Sym^*(V^*)/(u^k_\ell)$ is the principal ideal of $\Sym^*(V^*)$ generated by $\L_{n}^k$.
	Since $\L_{n}^k$ belongs to the ideal $I_{V^*,k}$, this map induces an inclusion of $\P$-modules $$\Sym^*(V^*)/I(V^*,k)\hookrightarrow \prod_{\ell\in \Line_{V^*}} \Sym^*(V^*)/(I(V^*,k)+(u^k_\ell)).$$
	The first part of the proposition now follows from Lemma \ref{ideal-rel}.
	
	For the second part, by induction we see that $\RR_{n,k}$ is embedded in a direct product of copies of $\F_q[x]/(x^{q^{n-1}k})\otimes \cdots \otimes \F_q[x]/(x^k)$. 
	By Cartan formula, it suffices now to prove that, for all $m\ge 0$, we have $\chi(\Pcal^i)(x^t)=0$ in $\F_q[x]/(x^{q^mr})$ 	whenever $t+qi>q^mr-1$. Suppose in contrary that $\chi(\Pcal^i)(x^t)\not=0$, then using $\chi$-trick and instability, we have up to a nonzero scalar $$x^{q^mr-1}=x^{q^mr-1-(q-1)i-t}\cdot \chi(\Pcal^i)(x^t)\equiv  x^t\cdot \Pcal^i(x^{q^mr-1-(q-1)i-t})=0$$
	modulo $\P$-decomposable elements in $\F_q[x]/(x^{q^mr})$. This contradicts the fact that the ``spike" $x^{q^mr-1}$ is $\P$-indecomposable in $\F_q[x]/(x^{q^mr})$. Indeed if $x^{q^mr-1}$ were $\P$-decomposable, then up to nonzero scalar it would be the image under some operation $\P^{p^a}$ of $x^{q^mr-1-p^a(q-1)}$ (noting that $\P$ is multiplicatively generated by all the $\P^{p^a}$, $a\ge 0$, and not by all the $\P^{q^a}$, $a\ge 0$ \cite[page 338]{LarrySmith}). But we have 
	$$\P^{p^a}(x^{q^mr-1-p^a(q-1)})=\binom{q^mr-1-p^a(q-1)}{p^a}x^{q^mr-1} =\binom{(q^mr-p^aq)+(p^a-1)}{p^a}x^{q^mr-1}.$$
Since $q^mr-p^aq>0$, we get $q^m>p^aq/r>p^a$, and so $(q^mr-p^aq)$ is divisible by $p^{a+1}$. It follows that $p^a$ does not appear in the $p$-adic expansion of $(q^mr-p^aq)+(p^a-1)$ and so the binomial coefficient $\binom{(q^mr-p^aq)+(p^a-1)}{p^a}$ is zero, completing the proof that $x^{q^mr-1}$ is $\P$-indecomposable. 
\end{proof}

 	\begin{rmk}The second part of Proposition \ref{emb} is equivalent to the fact that the $\P$-module
 		$\Sigma^d\Dual(\RR_{n,k})$ is unstable. Here given $M$ a $\P$-module, $\Dual(M)$ is the Spanier-Whitehead dual of $M$, which is defined by 
 		$$\begin{cases}
 		(\Dual M)^{-i}=\Hom_{\F_q}(M^i,\F_q),\qquad i\in \Z,\\
 		\theta(f)=f\circ (\chi(\theta)),\qquad f\in \Dual M, \theta\in \P.
 		\end{cases}$$
 	\end{rmk}

 We now present the proofs of the above lemmas.

\begin{proof}[Proof of Lemma \ref{Vn}] Recall that the map $\Ph:\F_q[x_1,\ldots,x_n]\to \F_q[[x_1,\ldots,x_n]]$, $f\mapsto \sum_{i\ge 0} \Ph^i(f)$, is a homomorphism of algebras and $\Ph(x)=x+x^q+x^{q^2}+\cdots$ for each linear form $x$. Here $\Ph^i=(-1)^i\chi(\P^i)$ and $\Ph=1+\Ph^1+\Ph^2+\cdots$. 
	
	Put 
	$P(x_1,\ldots,x_{n-1};y_1,\ldots,y_r):=\Ph\big( \prod_{i=1}^{n-1}(1-x_i^{q-1})\cdot  y_1^{q^s}\cdots y_r^{q^s}\big)$
	and let  $R(x_1,\ldots,x_{n-1};y_1,\ldots,y_r)$ denote its homogeneous component of  degree $q^{n-1+s}r$.
	Since $\prod_{i=1}^{n-1}(1-x_i^{q-1})=\sum_{i=0}^{n-1}(-1)^ie_i$, the lemma will follow from the identity $$R(x_1,\ldots,x_{n-1};y_1,\ldots,y_r)=\V(x_1,\ldots,x_{n-1},y_1)^{q^s}\cdots \V(x_1,\ldots,x_{n-1},y_r)^{q^s}.$$
	For this it is sufficient to prove that $R(x_1,\ldots,x_{n-1};y_1,\ldots,y_r)$ is null whenever each $y_j$ is a linear combination of $x_1,\ldots,x_{n-1}$. Since $P(x_1,\ldots,x_{n-1};y_1,\ldots,y_r)$ is linear in each $y_j$ and symmetric in $x_1,\ldots,x_{n-1}$ as well as in $y_1,\ldots,y_r$, we need only to prove that  $R(x_1,\ldots,x_{n-1};y_1,\ldots,y_r)$ is null if $y_1=x_{1}$. But since $\Ph(x_{1}(1-x_{1}^{q-1}))=\Ph(x_{1}-x_{1}^q)=x_{1}$, it follows that 
	\begin{eqnarray*}
		P(x_1,\ldots,x_{n-1};x_1,y_2,\ldots,y_r)=x_1(x_1+x_1^q+\cdots)^{q^s-1}	\prod_{i=2}^{n-1}\big(1-(x_i+x_i^q+\cdots)^{q-1}\big)\prod_{j=2}^r(y_j+y_j^{q}
		+\cdots )^{q^s}
		\\
		= x_1\prod_{k=0}^{s-1}(x_1^{q^k}+x_1^{q^{k+1}}+\cdots)^{q-1}
		\prod_{i=2}^{n-1}\big(1-(x_i+x_i^q+\cdots)^{q-1}\big)\prod_{j=2}^r(y_j^{q^s}+y_j^{q^{s+1}}
		+\cdots ).
	\end{eqnarray*}
	It is clear that a monomial occurring in this product is of the form $x_1^{1+i_1}x_2^{i_2}\cdots x_{n-1}^{i_{n-1}} y_2^{j_2}\cdots y_r^{j_r}$ where $\alpha(i_1)\le s(q-1)$, $\alpha(i_\ell)\in \{0,q-1\}$ for $2\le \ell \le n-1$, and $\alpha(j_\ell)=1$ for $2\le\ell \le r$. Here the function $\alpha$ is defined by $\alpha(a):=a_0+a_1+\cdots$ where $a=a_0+a_1q+a_2q^2+\cdots$ is the $q$-adic expansion of $a$. 
	It follows that if the degree of such a monomial is $q^{n-1+s}r$ then we must have
	\begin{eqnarray*}
		\alpha(q^{n-1+s}r-1)&=&\alpha(i_1+i_2+\cdots+i_{n-1}+j_2+\cdots+j_r)\\
		&\le& \alpha(i_1)+\alpha(i_2)+\cdots+\alpha(i_{n-1})+\alpha (j_2)+\cdots+\alpha(j_r)\\
		&\le& (n-2+s)(q-1)+(r-1).
	\end{eqnarray*}
	This is a contradiction since we have $q^{n-1+s}r-1=(q-1)(1+q+\cdots +q^{n-2+s})+(r-1)q^{n-1+s}$, which implies that $\alpha(q^{n-1+s}r-1)=(n-1+s)(q-1)+(r-1).$
	The lemma follows.
\end{proof}

\begin{proof}[Proof of Lemma \ref{ideal-rel}] Let $W$ be a linear complement of $\F_q\langle y\rangle$ in $V^*$. If $H$ is a hyperplane of $V^*$ not passing through $y$, then $\V_{H,V^*}=\L_{V^*}/\L_H$ is divisible by $y$. It follows that $I(V^*,k)+(y^k)$ is generated by $(y^k)$ and the $\V_{H,V^*}^k$'s where $H$ is a hyperplane passing through $y$. Let $H$ be such a hyperplane. Then $W':=H\cap W$ is a hyperplane in $W$ and $H=\F_q\langle y\rangle \oplus W'$.
Choose a basis $\{y_2,\ldots,y_{n}\}$ of $W$ such that $\{y_2,\ldots,y_{n-1}\}$ is a basis of $W'$. 
We have
$$\begin{vmatrix}
y & y_2 & \cdots & y_{n-1}\\
y^q & y_2^q & \cdots & y_{n-1}^q \\
\vdots & \vdots & \ddots & \vdots \\
y^{q^{n-1}} & y_2^{q^{n-1}} & \cdots &y_{n-1}^{q^{n-1}}
\end{vmatrix}\V_{H,V^*}=\begin{vmatrix}
y & y_2 & \cdots & y_n\\
y^q & y_2^q & \cdots & y_n^q \\
\vdots & \vdots & \ddots & \vdots \\
y^{q^{n-1}} & y_2^{q^{n-1}} & \cdots &y_n^{q^{n-1}}
\end{vmatrix}.$$
Noting that $1\le r \le q-1$, this gives 
$$\L_{W'}^q \V_{H,V^*}\equiv \L_W^q \mod{(y^r)},$$
and so $\V_{H,V^*}\equiv \V_{W',W}^q \mod{(y^r)}.$
From this we deduce that $\V_{H,V^*}^{q^s}\equiv \V_{W',W}^{q^{s+1}} \mod (y_1^{q^sr})$ and thus we get $$\V_{H,V^*}^{q^sr}\equiv \V_{W',W}^{q^{s+1}r} \mod (y_1^{q^sr}).$$ The result follows. 
\end{proof}



\begin{rmk} We note that, by putting $n=m+1$, in the simplest case where $r=1$ and $s=0$, this lemma gives: $$\V(x_1,\ldots,x_m,x)=(-1)^{m}\sum_{i=0}^{m}\chi(\P^{\frac{q^{m}-1}{q-1}-i})\big( x e_i \big),$$
	Comparing this with the formula defining the Dickson invariants $\Q_{m,0}, \ldots, \Q_{m,m-1}$ \cite{Dic11}:
	$$\V(x_1,\ldots,x_{m},x)=\sum_{j=0}^{m-1} (-1)^{m-j}\Q_{m,j}x^{q^j},$$
	we obtain 
	$$ \Q_{m,j} = \sum_{i=0}^{m}\chi(\P^{\frac{q^{m}-q^j}{q-1}-i})\big( e_i \big).$$
	Note that (\cite[Lemma 2.4]{MinhWalker}) if $f$ is of degree $s$ and $\alpha(r(q-1)+s)>s$ then $\Ph^r(f)=0$ (this is proved easily by letting $\Ph$ act on a product of $s$ linear form, and then using the definition of the function $\alpha$). Applying this to each term in the above expression of $\Q_{m,j}$, we have 
	\begin{eqnarray*}
		\alpha\bigg(\big(\frac{q^{m}-q^j}{q-1}-i\big)(q-1)+i(q-1)\bigg)=\alpha(q^m-q^j)=(m-j)(q-1)>i(q-1)\Longleftrightarrow m-j >i.
	\end{eqnarray*}
	It follows that the expression above for $\Q_{m,j}$ can be simplified to:
	$$ \Q_{m,j} = \sum_{i=m-j}^{m}\chi(\P^{\frac{q^{m}-q^j}{q-1}-i})\big( e_i \big).$$
	For $j=0$, we thus recover in a different way the formula for the top Dickson invariant $\Q_{m,0}$ in \cite[Theorem 1.1]{MinhWalker}:
	$$\Q_{m,0}=\chi(\P^{\frac{q^{m}-1}{q-1}-m})(e_m).$$
\end{rmk}

\section{Proof of Theorem \ref{cuspidal}}

In this section we prove the following:
\begin{thm}[Theorem \ref{cuspidal}]
	$\dim_{\F_q}\Quot^{q^{n-1}-n}(\Sym^*(V^*))=(q-1)(q^2-1)\cdots (q^{n-1}-1),$ $n\ge 2$.
\end{thm}

This is proved by considering a matroid complex $K$ which can be seen as an affine version of the complex $\Delta(V^*)$. To define $K$, let us fix a hyperplane $W$ in $V^*$ and let $E$ denote an affine space not containing the origin and parallel to $W$. 
Let $K$ be the simplicial complex on the vertex set $E$ in which a face is of the form $E\setminus F$ where $F$ affinely spans $E$, or equivalently, $F$ linearly spans $V^*$. 

We check easily that the facets of $K$ are complements in $E$ of affine bases of $E$ and a minimal non-face of $K$ is of the form $E\setminus H$ where $H$ is a hyperplane in the affine space $E$. Since a facet of $K$ has cardinality $d:=q^{n-1}-1$, $K$ is a $(d-1)$-dimensional simplicial complex.  

The Alexander dual $K^*$ of $K$ is the simplicial complex on the vertex set $E$ in which $F$ is a face if $F$ does not affinely span $E$. A facet of $K^*$ is thus a hyperplane in the affine space $E$.

\begin{lem}\label{K-matroid}
	$K$ is a matroid complex.
\end{lem}

\begin{proof}Let $B$ be the simplicial complex on $E$ whose faces are affinely independent subsets of $E$. It is clear that $E$ is a matroid complex. The dual of this matroid complex is $K$ and so $K$ is also a matroid complex. 
\end{proof}

 	\begin{lem}
 	The map $\iota:V \to \F_q\langle E\rangle$ defined by $\iota(v)=\sum_{\alpha\in E}\alpha(v)Y_\alpha$ is a monomorphism satisfying the property {\rm (*)} and its dual $\iota^*: \F_q^E\to V^*$ sends each basis element $X_\alpha\in \F_q^E$ to $\alpha\in V^*$.	
 	\end{lem}

\begin{proof}It is clear that $\iota$ is a monomorphism. Now given $E\setminus F$ a face of $K$, if $\iota(v)=\sum_{\alpha\in E}\alpha(v)Y_\alpha$ belongs to $\F_q\langle E\setminus F\rangle$, then one must have $\alpha(v)=0$ for all $\alpha\in F$, which implies that $v=0$ since $F$ is a spanning set of $V^*$. 
\end{proof}

Now we can apply Proposition \ref{lsop-dual} to study $\Sym^*(V^*)/\iota^*(I_K)$. The Stanley-Reisner ideal $I_K$ is generated by the monomials $\prod_{\alpha \in E\setminus H} X_\alpha$ where $H$ is running on the set of hyperplanes in the affine space $E$. It follows that the polynomials 
$\Phi_H:=\prod_{\alpha\in E\setminus H} \alpha$ generates the ideal $I:=\iota^*(I_K)$ of $\Sym^*(V^*)$. Note that $I$ is stable under the action of the affine subgroup $\Aff(E)$ of $\GL(V^*)$ and so the natural projection $\Sym^*(V^*)\twoheadrightarrow \Sym^*(V^*)/I$ is a map of $\Aff(E)$-modules.

\begin{prop} \label{Rmod-aff} For each $0\le i\le q^{n-1}-n$, the natural projection
	$\Sym^*(V^*)\twoheadrightarrow \Sym^*(V^*)/I$ induces an isomorphism of $\Aff(E)$-modules $\Quot^i(\Sym^*(V^*))\cong \Quot^i(\Sym^*(V^*)/I)$.  
\end{prop}

\begin{proof}
	Let $H$ be a hyperplane in the affine space $E$. We need to prove that $f\cdot \Phi_H$ is $\P$-decomposable whenever $$\deg(f)\le (q^{n-1}-n)-\deg(\Phi_H)=(q^{n-1}-n)-(q^{n-1}-q^{n-2})=q^{n-2}-n.$$
	Suppose $E=y+W$ and $H=y+W'$ where $y\in H$ and $W'$ is a hyperplane in $W$. Choose a basis $\{y_1,\ldots,y_{n-1}\}$ of $W$ such that $\{y_1,\ldots,y_{n-2}\}$ is a basis of $W'$. We have 
	\begin{eqnarray*}
	\Phi_H &=& \prod_{w'\in W', a\in \F_q^\times }(w'+ay_{n-1}+y) = \prod_{a\in \F_q^\times} \V(y_1,\ldots,y_{n-2},ay_{n-1}+y),
	\end{eqnarray*}
	and so by Lemma \ref{Vn}, 
	$$\Phi_H=(-1)^{n-2}\sum_{i=0}^{n-2}\chi(\P^{q^{n-2}-1-i})\big( e_i\cdot   Y \big),$$
	where $e_i:=e_i(y_1^{q-1},\ldots,y_{n-2}^{q-1})$ and $Y=\prod_{a\in \F_q^\times} (ay_{n-1}+y)$. By instability, $\P^{q^{n-2}-1-i}(f)=0$ for all $0\le i\le n-2$ and so the result follows from the $\chi$-trick. 
\end{proof}

\begin{prop}
	$\Sym^*(V^*)/I$ can be embedded as a $\P$-submodule into a direct product of $q^{n-1}$ copies of $\RR(W,q-1)$. 
\end{prop}

\begin{proof}
	The kernel of the natural map $\Sym^*(V^*)\to \prod_{y\in E} \Sym^*(V^*)/(y)$ is the principal ideal of $\Sym^*(V^*)$ generated by $\prod_{y\in E}y$. Since this generator belongs to $I$, this map induces an inclusion of $\P$-modules:
	$$\Sym^*(V^*)/I\hookrightarrow \prod_{y\in E} \Sym^*(V^*)/I+(y).$$
	It suffices now to prove that, for each $y\in E$, $I+(y)=I(W,q-1)+(y)$, and so $\Sym^*(V^*)/I+(y)\cong \RR(W,q-1)$. Indeed, given $H$ a hyperplane in $E$, if $y\not \in H$ then $\Phi_H$ is divisible by $y$. If $y\in H$ then $E=y+W$ and $H=y+W'$ where $W'$ is a hyperplane in $W$. Choose a basis $\{y_1,\ldots,y_{n-1}\}$ of $W$ such that $\{y_1,\ldots,y_{n-2}\}$ is a basis of $W'$. We have then
	\begin{eqnarray*}
	\Phi_H=\prod_{w'\in W', a\in \F_q^\times }(w'+ay_{n-1}+y)&\equiv& \prod_{a\in \F_q^\times} \V(y_1,\ldots,y_{n-2},ay_{n-1}) \pmod{(y)} \\ &=& -\V(y_1,\ldots,y_{n-2},y_{n-1})^{q-1} \pmod{(y)}\\
	\\ &=& -\V_{W',W}^{q-1} \pmod{(y)}.
	\end{eqnarray*}
	The result follows. 
\end{proof}

By Proposition \ref{emb}, we know that the top-degree module $\RR^{q^{n-1}-n}(W,q-1)$ is $\P$-indecomposable, and so the above proposition implies that the top-degree module $(\Sym^*(V^*)/I)^{q^{n-1}-n}$ is $\P$-indecomposable. Combining this with Proposition \ref{Rmod-aff} yields an isomorphism:
$$\Quot^{q^{n-1}-n}(\Sym^*(V^*))\cong (\Sym^*(V^*)/I)^{q^{n-1}-n}.$$
To compute the dimension of $(\Sym^*(V^*)/I)^{q^{n-1}-n}$, we need to compute the reduce Euler characteristic $\tilde{\chi}(K^*)$. Recall that the Alexander dual $K^*$ of $K$ is the simplicial complex on the vertex set $E$ in which a face is a set of points $F$ such that $F$ does not affinely span $E$. Folowing Lusztig \cite{Lusztig74}, let $S$ be the poset of all affine subspaces of $E$ other than $E$ (which is denoted by $S_I(E)$ in \cite{Lusztig74}). Then as in the linear case, the map of posets $f: p(K^*)\to S$ sending a face $F$ of $K^*$ to the affine subspace spanned by $F$ is a homotopy equivalence. This is because for each affine subspace $E'$ of $E$ other than $E$, the poset 
$f^{-1}(S_{\le E'})$ has $E'$ as its unique maximal element. The homotopy equivalence then again follows from the Quillen Fiber lemma.  
The reduced homology $\Ht_*(S;\F_q)$ is only nontrivial in degree $n-2$, and in this degree, $\Ht_{n-2}(S;\F_q)$, known as the \textit{affine Steinberg module} associated to $E$ \cite[1.14]{Lusztig74}, is of dimension $(q-1)(q^2-1)\cdots (q^{n-1}-1)$. It follows that $\tilde{\chi}(K^*)$ is up to sign equal to this dimension and so the proof of Theorem \ref{cuspidal} is complete.

\begin{rmk} By Proposition \ref{lsop-dual}, the elements $\prod_{y\in E\setminus B}\alpha$, where $B$ is an affine basis of $E$, span the $\F_q$-vector space $(\Sym^*(V^*)/I)^{q^{n-1}-n}$. It would be interesting to figure out a basis of this space from a shelling of the matroid complex $K$. Since $(q-1)(q^2-1)\cdots (q^{n-1}-1)$ is also the dimension of each complex cuspidal representation of $\GL_n$, it would also be interesting to study $(\Sym^*(V^*)/I)^{q^{n-1}-n}$ as a modular representation of $\GL_n$ and relate it to the discrete series constructed by Lusztig in \cite[Chapter 3]{Lusztig74}.  
\end{rmk}

\section{Steinberg summand of $\RR_{n,2}$ and Brown-Gitler modules}

In this section we prove Theorem \ref{decomposition}. Suppose $q=2$. We recall that $\RR_{n,2}=\F_2[x_1,\ldots,x_n]/I_{n,2}$, where 
$I_{n,2}$ is the ideal generated by the orbit of the action of $\GL_n$ on $\V_n^2$,  where $$\V_n=\prod_{\lambda_i\in \F_2}(\lambda_1x_1+\cdots+\lambda_{n-1}x_{n-1}+x_n).$$
Put $d=2(2^n-1)-n=\sum_{j=1}^n(2^j-1)$.
We need to prove that there is an isomorphism of $\A$-modules
$$\bigoplus_{j=0}^{n} \Sigma^{d-(2^j-1)}\BG(2^j-1) \xrightarrow{\cong} \RR_{n,2}\cdot \st_n.$$
Here $\BG(k)$ denotes the Brown-Gitler module 
$\BG(k):=\A/\A\langle \chi(\Sq^i) \mid 2i>k\rangle$. Note that $\Sigma^k\Dual \BG(k)$ is isomorphic to the unstable Brown-Gitler module $\J(k)$ ($\Dual$ denoting the Spanier-Whitehead dual) \cite[Proposition 5.4.4]{LZ87}.

We prove the above decomposition of $\RR_{n,2}\cdot \st_n$ by proving the following propositions.

\begin{prop}There is an $\A$-generating set $\{\alpha_0,\alpha_1,\ldots,\alpha_n\}$ with $|\alpha_j|=d- (2^j-1)$ for $\RR_{n,2}\cdot \st_n$. 
\end{prop}

\begin{proof}
	Recall that $\RR_{n,2}$ is a quotient of the polynomial algebra $\F_2[x_1,\ldots,x_n]$ which is trivial in degree greater than $d$. 	
	In \cite{Ino02}, M. Inoue proved that the Steinberg summand $\F_2[x_1,\ldots,x_n]\cdot \st_n$ is minimally $\A$-generated by the classes 
	$$\Sq^{2^{i_1}}\cdots \Sq^{2^{i_n}} (\frac{1}{x_1\ldots x_n})$$
	where $i_1>i_2>\cdots>i_n\ge 0$. 
	In degrees less than or equal to $d$, these corresponds to the sequences $(i_1,\ldots,i_n)=(2^n,\ldots, \widehat{2^j}, \ldots,1)$ with $0\le j\le n$. Projecting these generators to $\RR_{n,2}\cdot \st_n$ gives the desired result.
\end{proof}

\begin{rmk}The above generating set for $\RR_{n,2}$ is \textit{minimal} since we have $\Quot^i(\Sym_n)\cong \Quot^i(\RR_{n,2})$ in the range $0\le i\le d$ (Proposition \ref{Rmod}). However we do not need this fact to produce the surjective map in the following proposition.
\end{rmk}

\begin{prop}
Evaluating on the generators $\alpha_0,\alpha_1,\ldots,\alpha_n$ induces an epimorphism of $\A$-modules:
$$\bigoplus_{j=0}^{n} \Sigma^{|\alpha_j|}\BG(d-|\alpha_j|)=\bigoplus_{j=0}^{n} \Sigma^{d-(2^j-1)}\BG(2^j-1) \twoheadrightarrow \RR_{n,2}\cdot \st_n.$$
\end{prop}

\begin{proof}
	Evaluating on the generators $\alpha_0,\alpha_1,\ldots,\alpha_n$ yields an epimorphism of $\A$-modules:
	$$\bigoplus_{j=0}^{n} \Sigma^{|\alpha_j|}\A \twoheadrightarrow \RR_{n,2}\cdot \st_n, \qquad (\Sigma^{|\alpha_j|}\theta_j)_{0\le j\le n}\mapsto \sum_{j=0}^{n}\theta_j(\alpha_j).$$
	This map factors through the canonical projection $\bigoplus_{j=0}^{n} \Sigma^{|\alpha_j|}\A \twoheadrightarrow \bigoplus_{j=0}^{n} \Sigma^{|\alpha_j|}\A/\J_{d-|\alpha_j|}$ since we have 
	$\chi(\Sq^i)(\alpha_j)=0$ in $\RR_{n,2}$ 
	whenever $2i>d-|\alpha_j|$ (Proposition \ref{emb}). 
\end{proof}

%

\begin{prop}
	In each degree, the dimension of $\RR_{n,2}\cdot \st_n$ is greater than or equal to the dimension of $\bigoplus_{j=0}^{n} \Sigma^{d-(2^j-1)}\BG(2^j-1)$.
\end{prop}

\begin{proof}
Using the exact sequences (known as Mahowald's exact sequences) 
$$0\to \Sigma^{2^{j-1}}\BG(2^{j-1})\to \BG(2^j)\to \BG(2^j-1)\to 0,$$
we see that the Hilbert-Poincar\'e series of $\bigoplus_{j=0}^{n} \Sigma^{d-(2^j-1)}\BG(2^j-1)$ is the same as the Hilbert-Poincar\'e series of $\Sigma^{2^n-1-n}\BG(2^n)$. Since $\BG(2^n)$ has a basis $\{ \chi(\Sq^{i_1}\cdots \Sq^{i_n})\mid 2^{n-1}\ge i_1\ge 2i_2\ge \cdots \ge 2^{n-1}i_n\ge 0 \}$ \cite{BG73}, it follows that 
\begin{eqnarray*}
	P(\bigoplus_{j=0}^{n} \Sigma^{d-(2^j-1)}\BG(2^j-1),t)&=&\sum_{2^{n-1}\ge i_1\ge 2i_2\ge \cdots \ge 2^{n-1}i_n\ge 0} t^{2^n-1-n+i_1+i_2+\cdots +i_n}\\
	&=&\sum_{2^{n-1}\ge i_1\ge 2i_2\ge \cdots \ge 2^{n-1}i_n\ge 0}t^{(2^{n-1}+i_1)+(2^{n-2}+i_2)+\cdots +(1+i_n) -n}\\
	&=&\sum_{2^{n}\ge j_1\ge 2j_2\ge \cdots \ge 2^{n-1}j_n> 0}t^{j_1+j_2+\cdots +j_n -n}.
\end{eqnarray*}
For the Hilbert-Poincar\'e series of $\RR_{n,2}\cdot \st_n$, we prove that the set 
$$\{\Sq^{j_1}\cdots \Sq^{j_n}(\frac{1}{x_1\cdots x_n}) \mid 2^{n}\ge j_1\ge 2j_2\ge \cdots \ge 2^{n-1}j_n> 0 \}$$
is linearly independent in $\RR_{n,2}$. Note that each element in this set is fixed by the Steinberg idempotent $\st_n$ (\cite[Lemma 5.9]{MP83}). By Lemma \ref{ideal-rel}, we see that $I_{n,2}$ is contained in the ideal $(x_1^{2^n},\ldots,x_n^2)$. The linear independence then follows from the natural projection $$\RR_{n,2}\twoheadrightarrow \F_2[x_1,\ldots,x_n]/(x_1^{2^n},\ldots,x_n^2)$$ and the formula (\cite[Lemma 3.6]{MP83})
$$\Sq^{j_1}\cdots \Sq^{j_n}(\frac{1}{x_1\cdots x_n}) =x_1^{j_1-1}\cdots x_n^{j_n-1}+\text{monomials of lower order}$$
(in the lexicographical order starting at the left). The proposition is proved.
\end{proof}

\noindent\textbf{Acknowledgements.} This research is funded by Vietnam National Foundation for Science and Technology Development (NAFOSTED) under grant number 101.04-2019.10.

\def\cprime{$'$}




\end{document}